\documentclass[a4paper,11pt]{article}
\usepackage[pagewise]{lineno}
\usepackage{amsmath}
\usepackage{amssymb}
\usepackage{cases}
\usepackage{mathrsfs}
\usepackage{amsfonts}
\usepackage{amsthm}
\usepackage{pstricks, pst-node, pst-text, pst-3d,psfrag}
\usepackage{graphicx}
\usepackage{indentfirst}
\usepackage{enumerate}
\usepackage{subfigure}
\usepackage{cite}
\usepackage[colorlinks=true]{hyperref}
\hypersetup{urlcolor=blue, citecolor=red}

\theoremstyle{plain}
\newtheorem{thm}{Theorem}[section]
\newtheorem{prop}[thm]{Proposition}

\theoremstyle{definition}

\theoremstyle{remark}
\newtheorem{rmk}[thm]{\textbf{Remark}}

\numberwithin{equation}{section}

\makeatletter 
\@addtoreset{equation}{section}
\makeatother  

\newcommand\RR{\ensuremath{\mathbb{R}}}

\begin{document}

\title{Almost automorphy of minimal sets for $C^1$-smooth strongly monotone skew-product semiflows on Banach spaces\thanks{Supported by NSF of China No.11825106, 11771414 and 11971232, CAS Wu Wen-Tsun Key Laboratory of Mathematics, University of Science and Technology of China.}}

\setlength{\baselineskip}{16pt}

\author {
Yi Wang and Jinxiang Yao\thanks{Corresponding author: jxyao@mail.ustc.edu.cn (J. Yao).}
\\[2mm]
School of Mathematical Sciences\\
University of Science and Technology of China\\
Hefei, Anhui, 230026, P. R. China
}
\date{}
\maketitle

\begin{abstract}
We focus on the presence of almost automorphy in strongly monotone skew-product semiflows on Banach spaces. Under the $C^1$-smoothness assumption, it is shown that any linearly stable minimal set must be almost automorphic. This extends the celebrated result of Shen and Yi [Mem. Amer. Math. Soc. 136(1998), No. 647] for the classical $C^{1,\alpha}$-smooth systems. Based on this, one can reduce the regularity of the almost periodically forced differential equations and obtain the almost automorphic phenomena in a wider range.
\vskip 2mm

\textbf{Keywords}: Almost automorphy; Monotone skew-product semiflow; Principal Lyapunov exponents; Exponential Separation; $C^1$-smoothness.
\end{abstract}

\section{Introduction}
The notion of almost automorphy, which is a generalization to almost periodicity, was first introduced by Bochner \cite{Bochner1955} in a work of differential geometry. In the terminology of function theory, almost periodic and almost automorphic functions can be viewed as natural generalizations to the periodic ones in the strong and weak sense, respectively. From dynamical systems point of view,
Veech \cite{V63,V65,V67,V70} first introduced almost automorphic minimal flows. A compact flow $(Y,\mathbb{R})$ is called almost automorphic minimal if $Y$ is the closure of the orbit of an almost automorphic point. Here, a point $y\in Y$ is called almost automorphic if any net $\alpha'\subset \RR$
has a subnet $\alpha=\{t_n\}$ such that $T_\alpha y$, $T_{-\alpha}T_\alpha y$ exist and $T_{-\alpha}T_\alpha y=y$, where $T_\alpha$ is the generalized translation as $T_\alpha y=\lim_ny\cdot t_n$ provided that the limit exists (see Section \ref{section2}). Fundamental properties of almost automorphic functions/flows were further investigated in \cite{FB89,Te73,Te72}, etc.

Although an almost automorphic flow is a natural generalization of an almost periodic one, its topological and measure theoretical characterizations are different from an almost periodic one. For example, it may admit positive topological entropy (\cite{MP79}), it is not necessarily uniquely ergodic (\cite{Johnson81,Johnson82}) and its general measure theoretical characterization can be completely random (\cite{FB89}). Typical examples of almost automorphic minimal sets include the Toeplitz minimal sets in symbolic dynamics (\cite{DL98,MP79}), the Aubry-Mather sets on an annulus (\cite{Aubry82,Mather82}), and the Denjoy sets on the circle (\cite{Denjoy32}), etc. For other examples and details about almost automorphy, one may refer to \cite{Yi04,HuYi09} and references therein.

Almost autommorphy is fundamental and essential in almost periodic differential equations. As a matter of fact, in almost periodically forced differential equations, almost automorphic dynamics largely exist but almost periodicity needs not. For instance, Johnson (\cite{Johnson80,Johnson81A,Johnson81}) showed the almost automorphy in linear scalar ODEs (or $2$-dimensional linear ODE systems) with almost periodic coefficients. For almost periodically-forced scalar parabolic equations, a series work of Shen and Yi (\cite{ShYi95asymtpotic,ShYi95dynamic,ShYi95on,ShYi96}) discovered the almost automorphy of any minimal sets for separated boundary conditions (e.g., Dirichlet, Neumann and Robin Types); while, for periodic boundary conditions, the almost automorphic dynamics was recently systemically studied by Shen et. al in \cite{SWZ16,SWZ19,SWZ20a}.

Monotone skew-product (semi)flows are another abundant and important sources of almost automorphic dynamics. The study of monotone skew-product systems is a natural extension (to nonautonomous or external-forced systems) of the pioneering work by M. W. Hirsch (\cite{H84,H88a,H82,H85}) on monotone dynamical systems (see also Matano \cite{M86}). Large quantities of mathematical models of ordinary, functional and partial differential equations or difference equations can generate monotone dynamical systems. One may refer to the monographs and reviews  \cite{HS05,P02,S95,S17,ShYi98,Chue02} for more details.
Hirsch showed that the generic precompact orbit of a strongly monotone dynamical
system approaches the set of equilibria (referred as generic quasi-convergence). For $C^1$-smooth strongly monotone semiflows, the improved generic convergence was obtained by Pol\'{a}\v{c}ik \cite{P89} and Smith and Thieme \cite{ST91}. For strongly monotone discrete-time systems (mappings),
which are usually the Poincar\'{e} mappings associated with periodically forced differential equations, Pol\'{a}\v{c}ik and Tere\v{s}\v{c}\'{a}k \cite{PT92} proved that the
generic convergence to cycles occurs provided that the mapping $F$ is of $C^{1,\alpha}$-class (i.e., $F$ is a $C^1$-map with a locally $\alpha$-H\"{o}lder derivative $DF$, $\alpha\in (0,1]$). For the lower regularity of $F$, Tere\v{s}\v{c}\'{a}k \cite{T94} and Wang and Yao \cite{YJ20} succeeded in using different approaches to prove the generic convergence to cycles for $C^1$-smooth strongly monotone discrete-time systems.

Shen and Yi \cite{ShYi98} first discovered that almost automorphic phenomena largely exist
in strongly monotone skew-product semiflows $\Pi(x,y,t)=(u(x,y,t),y\cdot t)$ on $X\times Y$, $t\geq0$, where $X$ is a Banach space, $(Y,\mathbb{R})$ is a minimal and distal flow. More precisely, under the assumption that $u$ is $C^{1,\alpha}$ in $x\in X$, they \cite{ShYi98} studied the lifting dynamics on minimal sets of the strongly monotone skew-product semiflow $\Pi$, and proved that a linearly stable minimal set must be almost automorphic and that the generic convergence property failed in almost periodic systems even within the category of almost automorphy.  Their results have also been applied to show the existence of almost automorphic dynamics in a large class of almost periodic ordinary, functional and parabolic differential equations. Based on Shen and Yi's work, Obaya and his collaborators \cite{obaya05,obaya13book,obaya13} systematically analyzed the occurrence of almost automorphic dynamics in monotone skew-product semiflows with applications to functional differential equations.

The approach in \cite{ShYi98} is based on establishment of the exponential separation (see, e.g. \cite{PT93,Mi91}) along the minimal sets of $\Pi$, as well as
the idea and techniques for the construction of invariant measurable families of submanifolds in the so called Pesin's Theory (see \cite{Pesin76}); and hence, the regularity of $\alpha$-H\"{o}lder continuity of the $x$-derivative of $u$ cannot be dropped in \cite{ShYi98}.

In the present paper, we shall focus on the presence of almost automorphy in $C^1$-smooth strongly monotone skew-product semiflows.  Motivated by our recent work in \cite{YJ20}, we will extend the celebrated result of Shen and Yi \cite{ShYi98} by showing that any linearly stable minimal set must be almost automorphic for $C^1$-smooth strongly monotone skew-product semiflows. Based on our result, one can reduce the regularity of the almost periodically forced equations (inculding ODEs, parabolic equations and delay equations) investigated in \cite[Part III]{ShYi98}, and obtain the almost auotmorphic phenomena in a wider range.

As mentioned above, due to the lack of the $\alpha$-H\"{o}lder continuity, the Pesin's Theory with the Lyapunov exponents arguments in \cite{ShYi98} can not work any more. Inspired by \cite{T94,YJ20}, our new approach is to introduce a continuous cocycle over the Cartesian square $K\times K$ of the linearly stable minimal set $K$ rather than $K$ itself, and to construct a bundle map $T$ as the hybrid function of the $x$-derivative of $u$ along $K\times K$. Together with the exponential separation on $K$ with a novel ``internal growth control" property (see Proposition \ref{prop ES-contin}(v)) and a time-discretization technique to the skew-product semiflow, we accomplish our approach by proving the crucial Propositions \ref{Prop substitute infty} and \ref{Prop substitute finite}, which enables us to reduce the regularity of the systems and obtain the almost automorphy of the minimal sets.

This paper is organized as follows. In Section \ref{section2}, we agree on some notations, give relevant
definitions and preliminary results. We further present the exponential separation theorem (see Proposition \ref{prop ES-contin}) with the novel additional ``internal growth control'' property along principal bundles in Proposition \ref{prop ES-contin}(v), which turns out to be crucial for the proof of our main result. In section \ref{section3}, we state our main
results and give their proofs.

\section{Notations and Preliminary Results}\label{section2}

In this section, we first summarize some preliminary materials involved with topological dynamics which will appear throughout the paper.

Let $(Y,d_{Y})$ be a compact metric space, and $\sigma:Y\times\mathbb{R}\to Y$, $(y,t)\mapsto y\cdot t$ be a continuous flow on $Y$, denoted by $(Y,\sigma)$ or $(Y,\mathbb{R})$. A subset $M\subset Y$ is \emph{invariant} if $\sigma_{t}M=M$, for each $t\in\mathbb{R}$. A non-empty compact invariant set $M\subset Y$ is called \emph{minimal} if it contains no non-empty, proper, closed invariant subset. We say that $(Y,\mathbb{R})$ is minimal if $Y$ itself is a minimal set.

Let ${\mathbb{R}}^{+}$, ${\mathbb{R}}^{-}$ denote the nonnegative, nonpositive reals, respectively. Points $y_{1},y_{2}\in Y$ are called (\emph{positively, negatively}) \emph{distal}, if $\mathop{\inf}\limits_{t\in\mathbb{R}(t\in{\mathbb{R}}^{+},t\in{\mathbb{R}}^{-})}d_{Y}(y_{1}\cdot t,y_{2}\cdot t)>0$. We say that $y_{1},y_{2}$ are (\emph{positively, negatively}) \emph{proximal} if they are not (positively, negatively) distal. A point $y\in Y$ is said to be a \emph{distal point} if it is only proximal to itself. Moreover, $(Y,\mathbb{R})$ is a \emph{distal flow} if every point in $Y$ is a distal point. The (\emph{positive, negative}) \emph{proximal relation} $P(Y)(P_{+}(Y),P_{-}(Y))$ is a subset of $Y\times Y$ defined as follows: $P(Y)(P_{+}(Y),P_{-}(Y))=\{(y_{1},y_{2})\in Y\times Y|y_{1},y_{2}\text{ are (positively, negatively) proximal}\}$. $P(Y)$ is clearly invariant, reflexive and symmetric but not transitive in general.

\begin{prop}\label{prop P=P+=P-}
{\rm (\cite[Part I, Corollary 2.8]{ShYi98}).} Suppose that $P(Y)$ is an equivalence relation. Then $P(Y)=P_{+}(Y)=P_{-}(Y)$.
\end{prop}

For $y\in Y$ and a net $\alpha=\{t_{n}\}$ in $\mathbb{R}$, we define $T_{\alpha}y:=\lim_ny\cdot t_n$, provided that the limit exists. $(Y,\mathbb{R})$ is called \emph{almost periodic} if any nets $\alpha^{\prime}$, $\beta^{\prime}$ in $\mathbb{R}$ have subnets $\alpha$, $\beta$ such that $T_{\beta}y$, $T_{\alpha}T_{\beta}y$, $T_{\alpha+\beta}y$ exist and $T_{\alpha}T_{\beta}y=T_{\alpha+\beta}y$ for all $y\in Y$, where $\alpha+\beta=\{t_{n}+s_{n}\}$ if $\alpha=\{t_{n}\}$, $\beta=\{s_{n}\}$. An almost periodic flow is necessarily distal (see, e.g. \cite{ShYi98}). A point $y\in Y$ is an \emph{almost automorphic point} if any net $\alpha^{\prime}$ in $\mathbb{R}$ has a subnet $\alpha=\{t_{n}\}$ such that $T_{\alpha}y$, $T_{-\alpha}T_{\alpha}y$ exist and $T_{-\alpha}T_{\alpha}y=y$, where $-\alpha=\{-t_{n}\}$. A flow $(Y,\mathbb{R})$ is \emph{almost automorphic} if there is an almost automorphic point $y_{0}\in Y$ with dense orbit. An almost automorphic flow is necessarily minimal (see, e.g. \cite{ShYi98}).

A \emph{flow homomorphism} from another continuous flow $(Z,\mathbb{R})$ to $(Y,\mathbb{R})$ is a continuous map $\phi:Z\to Y$ such that $\phi(z\cdot t)=\phi(z)\cdot t$ for all $z\in Z$, $t\in\mathbb{R}$. An onto flow homomorphism is called a \emph{flow epimorphism} and an one to one flow epimorphism is referred to as a \emph{flow isomorphism}. If $\phi$ is an epimorphism, then $(Z,\mathbb{R})$ is said to be an \emph{extension} of $(Y,\mathbb{R})$. An epimorphism $\phi$ is called an \emph{$N$-1 extension} for some integer $N\geq1$, if $card({\phi}^{-1}(y))=N$ for all $y\in Y$. Let $\phi:(Z,\mathbb{R})\to(Y,\mathbb{R})$ be a homomorphism of minimal flows, then $\phi$ is an \emph{almost automorphic extension} if there is a $y_{0}\in Y$ such that $card(\phi^{-1}(y_{0}))=1$. Then, actually $\phi$ is an \emph{almost 1-1 extension}, i.e., $\{y\in Y|card({\phi}^{-1}(y))=1\}$ is a residual subset of $Y$. A minimal flow $(Z,\mathbb{R})$ is \emph{almost automorphic} if and only if it is an almost automorphic extension of an almost periodic minimal flow $(Y,\mathbb{R})$ (see \cite{V65} or \cite[Part I, Theorem 2.14]{ShYi98}).

\begin{prop}\label{prop N-1}
{\rm (\cite{SS77} or \cite[Part I, Theorem 2.12]{ShYi98}).} Let $\phi:(Z,\mathbb{R})\to(Y,\mathbb{R})$ be a homomorphism of distal flows, where $(Y,\mathbb{R})$ is minimal. If there is $y_{0}\in Y$ with $card(\phi^{-1}(y_{0}))=N$, then the following holds: 1) $\phi$ is an $N$-1 extension; 2)$(Z,\mathbb{R})$ is almost periodic if and only if $(Y,\mathbb{R})$ is.
\end{prop}

Given a continuous flow $(Y,\mathbb{R})$ and a Banach space $X$, a continuous \emph{skew-product semiflow} $\Pi:X\times Y\times \mathbb{R}^{+}\to X\times Y$ is defined as:
\begin{equation}\label{df of skew-product}
\Pi(x,y,t)=(u(x,y,t),y\cdot t),\quad (x,y)\in X\times Y,\ t\in\mathbb{R}^{+},
\end{equation}
where $\Pi(\cdot,\cdot,t)$ can also be written as $\Pi_{t}(\cdot,\cdot)$, for all $t\in\mathbb{R}^{+}$ and satisfies (i) $\Pi_{0}=$Id and (ii) the \emph{cocycle property}: $u(x,y,t+s)=u(u(x,y,s),y\cdot s,t)$, for each $(x,y)\in X\times Y$ and $t,s\in\mathbb{R}^{+}$. We denote $p:X\times Y\to Y; (x,y)\mapsto y$ as the natural projection. A subset $M\subset X\times Y$ is called \emph{positively invariant} if $\Pi_{t}(M)\subset M$ for all $t\in\mathbb{R}^{+}$. A compact positively invariant set $K\subset X\times Y$ is \emph{minimal} if it does not contain any other nonempty compact positively invariant set than itself.

A \emph{flow extension} of a skew-product semiflow $(X\times Y,\Pi,\mathbb{R}^{+})$ is a skew-product flow $(X\times Y,\tilde{\Pi},\mathbb{R})$ such that $\tilde{\Pi}(x,y,t)=\Pi(x,y,t)$, for each $(x,y)\in X\times Y$ and $t\in\mathbb{R}^{+}$. A compact positively invariant subset is called admits a \emph{flow extension} if the semiflow restricted to it does. Actually, a compact positively invariant set $K\subset X\times Y$ admits a flow extension if every point in $K$ admits a unique backward orbit which remains inside the set $K$ (see \cite[Part II]{ShYi98}).

In this work, we need $C^{1}$-smoothness of the skew-product semiflow $\Pi$. Precisely, the skew-product semiflow $\Pi$ in (\ref{df of skew-product}) is said to be of \emph{class ${C}^{1}$ in $x$}, meaning that $u_{x}(x,y,t)$ exists for any $t>0$ and any $(x,y)\in X\times Y$; and for each fixed $t>0$, the map $(x,y)\mapsto u_{x}(x,y,t)\in\mathcal{L}(X)$ is continuous on any compact subset $K\subset X\times Y$; and moreover, for any $v\in X$, $u_{x}(x,y,t)v\to v$ as $t\to 0_{+}$ uniformly for $(x,y)$ in compact subsets of $X\times Y$.

Let $K\subset X\times Y$ be a compact, positively invariant set which admits a flow extension. For $(x,y)\in K$, we define the \emph{Lyapunov exponent} $\lambda(x,y)$ as $\lambda(x,y)=\mathop{\limsup}\limits_{t \to +\infty}\frac{\ln\|u_{x}(x,y,t)\|}{t}$. The number $\lambda_{K}=\mathop{\sup}\limits_{(x,y)\in K}\lambda(x,y)$ is called the \emph{principal Lyapunov exponent} on $K$. If $\lambda_{K}\leq0$, then $K$ is said to be \emph{linearly stable}.

\begin{prop}\label{Prop norm growth control}
{\rm (\cite[PartII, Corollary 4.2]{ShYi98}).} Assume that $(Y,\mathbb{R})$ is minimal and $\Pi$ is of class ${C}^{1}$ in $x$. Assume also that $K\subset X\times Y$ is a compact, positively invariant set which admits a flow extension; moreover, $K$ is \emph{linearly stable}. Then for any $\epsilon>0$, there is a $C_{\epsilon}>0$ such that $\|u_{x}(x,y,t)\|\leq C_{\epsilon}e^{\varepsilon t}$, for all $t\geq0$ and $(x,y)\in K$.
\end{prop}

A closed convex subset $C\subset X$ is called a cone of $X$ if $\lambda C\subset C$ for all $\lambda>0$ and $C\cap(-C)=\{0\}$. We call $(X,C)$ a \emph{strongly ordered} Banach space if $C$ has nonempty interior ${\rm Int}C$. Let $X^{*}$ be the dual space of $X$. $C^{*}$:=\{$l\in X^{*}:l(v)\geq 0$ for any $v\in C$\} is called the dual cone of $C$. If ${\rm Int}C\neq\emptyset$, then $C^{*}$ is indeed a closed convex cone in $X^{*}$ (see \cite{D85}). Let $C_{s}^{*}=\{l\in C^{*}:l(v)>0,\text{ for any }v\in C\backslash\{0\}\}$. A bounded linear operator $L:X\to X$ is \emph{strongly positive} if $Lv\gg0$ whenever $v>0$.

Let $(X,C)$ be a \emph{strongly ordered} Banach space. A closed set $O_{+}(X,Y):=\{((x_{1},y),(x_{2},y))|x_{1}-x_{2}\in C\}$ induces a (strong) partial ordering `$\geq$' on each fiber $p^{-1}(y)\ (y\in Y)$ as follows: $(x_{1},y)\geq(x_{2},y)$ if $((x_{1},y),(x_{2},y))\in O_{+}(X,Y)$; $(x_{1},y)>(x_{2},y)$ if $(x_{1},y)\geq(x_{2},y),(x_{1},y)\neq(x_{2},y)$; $(x_{1},y)\gg(x_{2},y)$ if $((x_{1},y),(x_{2},y))\in \text{Int}O_{+}(X,Y)$, i.e., $x_{1}-x_{2}\in\text{Int}C$. $O_{-}(X,Y)$ is the reflection of $O_{+}(X,Y)$, that is, $O_{-}(X,Y)=\{((x_{1},y),(x_{2},y))|((x_{2},y),(x_{1},y))\in O_{+}(X,Y)\}$. The set $O(X,Y)=O_{+}(X,Y)\cup O_{-}(X,Y)$ is referred to as the \emph{order relation}, that is, $(x_{1},y_{1}),(x_{2},y_{2})$ are \emph{ordered} if and only if $y_{1}=y_{2}=y$ and $((x_{1},y_{1}),(x_{2},y_{2}))\in O(X,Y)$. The order relation on a minimal subset $K\subset X\times Y$ is defined as $O(K)=\{((x_{1},y),(x_{2},y))|(x_{1},y),(x_{2},y)\in K\ \text{and}\ x_{1}-x_{2}\in \pm C\}$.

The skew-product semiflow $\Pi$ is called \emph{strongly order preserving} if $\Pi(x_{1},y,t)\gg\Pi(x_{2},y,t)$ whenever $(x_{1},y)>(x_{2},y)$ and $t>0$. We say that $\Pi$ is \emph{strongly monotone} if $u_{x}(x,y,t)$ is a strongly positive operator for any $(x,y)\in X\times Y,\ t>0$. Clearly, by virtue of \cite[PartII, Theorem 4.3]{ShYi98}, a strongly monotone skew-product semiflow must be a strongly order preserving skew-product semiflow.

\begin{prop}\label{Prop residual set unordered}
Assume that $(Y,\mathbb{R})$ is minimal and $\Pi$ is strongly order preserving, and let $K\subset X\times Y$ be a minimal set of  which admits a flow extension. Then

\textnormal{(i)} there is a residual and invariant set $Y_{0}\subset Y$ such that for any $y\in Y_{0}$, no two elements on $K\cap p^{-1}(y)$ are ordered;

\textnormal{(ii)} If $(x_{1},y),(x_{2},y)\in K$ are ordered, then they are proximal, that is, the order relation implies the proximal relation on $K$.
\end{prop}

\begin{proof}
See \cite[PartII, Theorem 3.2 and Corollary 3.3]{ShYi98}.
\end{proof}

Before ending this section, we present the following exponential separation theorem for homeomorphisms. One may refer to \cite{PT93,Mi91,Mi94,JS} for more details and applications of this theorem with the standard items (i)-(iii). Here, we emphasize a novel ``internal growth control'' property along the principal bundles obtained in item (v) of the following proposition, which turns out to be crucial for the proof of our main results in the next section. A weaker version of such ``internal growth control'' property was obtained in \cite{T94,YJ20} for exponential separation for continuous maps.

\begin{prop}\label{prop ES-contin}
{\rm (Exponential Separation Theorem).} Let $(X,C)$ be a strongly ordered Banach space, $F: E \to E$ is a homeomorphism of a compact metric space $E$, $T$ is a continuous family of operators $\{T_{x}\in L(X,X):x\in E\}$, and for any $x \in E$, $T_{x}$ is a compact and strongly positive operator, then there exist one dimensional continuous bundles $E\times X_{1x}$ and $E\times X_{1x}^{*}$ such that:

\textnormal{(i)} $X_{1x}$=span$\{v_{x}\}$ and $X_{1x}^{*}$=span$\{l_{x}\}$, where $\|v_{x}\|=1=\|l_{x}\|$, $v_{x}\gg0$, $l_{x}\in C_{s}^{*}$, and both $l_{x}$ and $v_{x}$ depend continuously on $x\in E$.

\textnormal{(ii)} $T_{x}X_{1x}=X_{1Fx}$, $T_{x}^{*}X_{1Fx}^{*}=X_{1x}^{*}$.

\textnormal{(iii)} There are constants $M>0$ and $0<\gamma<1$ such that
\begin{equation}\label{3.2}
\|T_{x}^{n}w\|\leq M\gamma^{n}\|T_{x}^{n}v_{x}\|,
\end{equation}
for all $x\in E$, $n\geq1$ and $l_{x}(w)=0$ with $\|w\|$=1, where $T_{x}^{n}=T_{F^{n-1}x}\circ T_{F^{n-2}x}\circ\cdot\cdot\cdot\circ T_{Fx}\circ T_{x}$.

\textnormal{(iv)} If $x\in E, u\in X$ with $l_{x}(u)>0$, then $T_{x}^{n}u\in{\rm Int}C$ for all $n$ sufficiently large.

\textnormal{(v) (Internal growth control along principal bundles)} For any $\epsilon>0$, there is a constant $\delta_{1}>0$ such that, for any $\delta\in[0,\delta_{1}]$, $x,y\in E$, $m\geq 1$ with $d_{E}(F^{i}x,F^{i}y)<\delta$, $0\leq i\leq m$, we have
\begin{equation}\label{3.1}
\|T_{y}^{i}v_{y}\|\leq (1+\epsilon)^{i}\|T_{x}^{i}v_{x}\|,
\end{equation}
 for all $1\leq i\leq m$.
\end{prop}

\begin{proof}
For the proof of the standard items (i)-(iii), we refer to \cite{PT93}. Here we give the proof of (iv)-(v).

(iv). Decompose $u$ by $u=v+w$, with $v=\frac{l_{x}(u)}{l_{x}(v_{x})}v_{x}$, $l_{x}(w)=0$. Then we have
\begin{alignat*}{2}
\|v_{F^{n}x}-\frac{T_{x}^{n}u}{\|T_{x}^{n}u\|}\|&\leq\|v_{F^{n}x}-\frac{T_{x}^{n}v}{\|T_{x}^{n}v\|}\|+\|\frac{T_{x}^{n}v}{\|T_{x}^{n}v\|}-\frac{T_{x}^{n}(v+w)}{\|T_{x}^{n}(v+w)\|}\|\\
&\overset{\textnormal{(ii)}}{=}0+\|\frac{T_{x}^{n}v}{\|T_{x}^{n}v\|}-\frac{T_{x}^{n}(v+w)}{\|T_{x}^{n}(v+w)\|}\|\\
&\overset{\textnormal{(iii)}}{\to}0,\text{ as }n\to\infty.
\end{alignat*}
Since $\{v_{x}:x\in E\}$ is a compact subset of ${\rm Int}C$ by (i), $T_{x}^{n}u\in{\rm Int}C$ for all $n$ sufficiently large. This proves (iv).

(v). Since $T_{x}v_{x}$ continuously depends on $x\in E$, $\{T_{x}v_{x}:x\in E\}$ is a compact subset of ${\rm Int}C$. Then there exists a constant $r>0$ such that $\|T_{x}v_{x}\|>r$, for any $x\in E$. For any $\epsilon>0$, by $T_{x}v_{x}$ uniformly continuously depends on $x\in E$, there exists a constant $\delta_{1}>0$ such that $\|T_{x}v_{x}-T_{x^{\prime}}v_{x^{\prime}}\|\leq\epsilon r<\epsilon\|T_{x^{\prime}}v_{x^{\prime}}\|$, for any $x,x^{\prime}\in E$ with $d_{E}(x,x^{\prime})<\delta_{1}$. Therefore, for any $\delta\in[0,\delta_{1}]$, $x,y\in E$, $m\geq 1$ with $d_{E}(F^{i}x,F^{i}y)<\delta$, $0\leq i\leq m$, we have
\begin{alignat*}{2}
&\frac{\|T_{y}^{i}v_{y}\|}{\|T_{x}^{i}v_{x}\|}=
\frac{\|T_{F^{i-1}y}v_{F^{i-1}y}\|\cdots\|T_{Fy}v_{Fy}\|\cdot\|T_{y}v_{y}\|}{\|T_{F^{i-1}x}v_{F^{i-1}x}\|\cdots\|T_{Fx}v_{Fx}\|\cdot\|T_{x}v_{x}\|}
<(1+\epsilon)^{i} , \quad 1\leq i\leq m.
\end{alignat*}
This proves (v).
\end{proof}

\section{Main Results and Proofs}\label{section3}

In this section, our standing hypotheses are as follows:
\vskip 2mm

\noindent \textbf{(H1)} $(Y,\mathbb{R})$ is minimal and distal, and $(X,C)$ is a strongly ordered Banach space.
\vskip 1mm

\noindent \textbf{(H2)} $\Pi$ is a strongly monotone skew-product semiflow on $X\times Y$ of class ${C}^{1}$ in $x$.
\vskip 2mm

\noindent \textbf{(H3)} $K\subset X\times Y$ is a minimal set which admits a flow extension.
\vskip 2mm

Now we state our main results on the almost automorphy of the minimal set $K$.

\begin{thm}\label{Theorem main th}
Assume that \textnormal{(H1)-(H3)} hold. Assume also the following:

\textnormal{(i)} There is $\tau>0$ such that $u_{x}(x,y,\tau)$ is compact for all $(x,y)\in\hat{K}$, where $\hat{K}=\{(sx_{1}+(1-s)x_{2},y):(x_{1},y),(x_{2},y)\in K\ \textnormal{and}\ s\in[0,1]\}$.

\textnormal{(ii)} $K$ is linearly stable.

\noindent Then there is a minimal flow $(\tilde{Y},\mathbb{R})$ and flow homomorphisms
$$p^{*}:(K,\mathbb{R})\to(\tilde{Y},\mathbb{R})\quad and \quad\tilde{p}:(\tilde{Y},\mathbb{R})\to(Y,\mathbb{R})$$
such that $(\tilde{Y},\mathbb{R})$ is distal, $\tilde{p}$ is an $N$-1 extension for some integer $N\geq1$, $p^{*}$ is an almost 1-1 extension and $p=\tilde{p}\circ p^{*}$, where $p:K\to Y$ denotes the natural projection. Moreover, if $(Y,\mathbb{R})$ is almost periodic, then $(K,\mathbb{R})$ is almost automorphic.
\end{thm}

\begin{rmk}\label{remark3.0}
Under the assumption that $u$ is $C^{1,\alpha}$ in $x$, Shen and Yi \cite[PartII, Theorem 4.5]{ShYi98} proved that a linearly stable minimal set must be almost automorphic. As we mentioned in the introduction, the approach in \cite{ShYi98} is based on the idea and technique of construction of invariant measurable families of submanifolds in the so called Pesin's Theory (see \cite{Pesin76}). So, the regularity of $\alpha$-H\"{o}lder continuity of the $x$-derivative of $u$ cannot be dropped in \cite{ShYi98}. With the help of the exponential separation on $K\times K$ with the ``internal growth control'' property along the principal bundles and a time-discretization technique, we succeed in reducing the regularity.
\end{rmk}
\vskip 2mm

In the following, we will focus on the proof of Theorem \ref{Theorem main th}. Before we proceed further, we give the following two crucial propositons:

\begin{prop}\label{Prop substitute infty}
Let $K$ be as in Theorem \ref{Theorem main th}. Then there is a $\delta_{0}>0$ such that if $(x_{3},\tilde{y}),(x_{4},\tilde{y})\in K$ satisfies $\|x_{3}-x_{4}\|<\delta_{0}$ and $u(x_{3},\tilde{y},t),u(x_{4},\tilde{y},t)$ are not ordered (that is, $u(x_{3},\tilde{y},t)-u(x_{4},\tilde{y},t)\notin\pm C$) for all $t\geq0$, then
\begin{equation}\label{infty estimate}
\|u(x_{3},\tilde{y},t)-u(x_{4},\tilde{y},t)\|\to 0, \text{ as } t\to+\infty.
\end{equation}
\end{prop}

\begin{proof}
We write
$$K_{1}=\{((x_{1},y),(x_{2},y)):\ (x_{1},y),(x_{2},y)\in K\},$$
on which the metric is defined as
$$d_{K_{1}}(((x_{1},y),(x_{2},y)),((x_{1}^{\prime},y^{\prime}),(x_{2}^{\prime},y^{\prime})))=\sqrt{(d_{K}((x_{1},y),(x_{1}^{\prime},y^{\prime})))^{2}+(d_{K}((x_{2},y),(x_{2}^{\prime},y^{\prime})))^{2}},$$
for $((x_{1},y),(x_{2},y)),((x_{1}^{\prime},y^{\prime}),(x_{2}^{\prime},y^{\prime}))\in K_{1}$, where $d_{K}((x_{i},y),(x_{i}^{\prime},y^{\prime}))=\sqrt{\|x_{i}-x_{i}^{\prime}\|^{2}+d_{Y}(y,y^{\prime})^{2}}$, $i=1,2$.
Clearly, $(K_{1},d_{K_{1}})$ is a compact metric space. We define the continuous  map
$$F_{1}:K_{1}\to K_{1};\text{  }((x_{1},y),(x_{2},y))\mapsto F_{1}((x_{1},y),(x_{2},y))\triangleq(\Pi(x_{1},y,\tau),\Pi(x_{2},y,\tau)),$$
for any $((x_{1},y),(x_{2},y))\in K_{1}.$ Since $K$ admits a flow extension, $F_{1}$ is a homeomorphism.
Define the bundle map $T$ as a hybrid function as:
$$T_{((x_{1},y),(x_{2},y))}=\int_0^1{u_{x}(sx_{1}+(1-s)x_{2},y,\tau)ds},\text{  }((x_{1},y),(x_{2},y))\in K_{1}.$$
Recall that $u_{x}(x,y,\tau)$ is strongly positive and continuous in $(x,y)\in\hat{K}$. Then, for each $((x_{1},y),(x_{2},y))\in K_{1}$, $T_{((x_{1},y),(x_{2},y))}$ is a strongly positive linear operator on $X$ and $T_{((x_{1},y),(x_{2},y))}$ continuously depends on $((x_{1},y),(x_{2},y))\in K_{1}$. Moreover, together with the fact that $u_{x}(x,y,\tau)$ is compact for all $(x,y)\in\hat{K}$, we have for each $((x_{1},y),(x_{2},y))\in K_{1}$, $T_{((x_{1},y),(x_{2},y))}$ is a compact operator on $X$ . Furthermore, one can obtain
\begin{equation}\label{cocycle-1}
T_{((x_{1},y),(x_{1},y))}^{n}=u_{x}(x_{1},y,n\tau)
\end{equation}
and
\begin{equation}\label{cocycle-2}
T_{((x_{1},y),(x_{2},y))}^{n}(x_{1}-x_{2})=u(x_{1},y,n\tau)-u(x_{2},y,n\tau),
\end{equation}
for any $((x_{1},y),(x_{2},y))\in K_{1}$ and $n\in\mathbb{N}$. Here, $T_{((x_{1},y),(x_{2},y))}^{n}=T_{F_{1}^{n-1}((x_{1},y),(x_{2},y))}\circ\cdots\circ T_{F_{1}((x_{1},y),(x_{2},y))}\circ T_{((x_{1},y),(x_{2},y))}$. In fact, (\ref{cocycle-1}) is direct from the co-cycle property of $\Pi$. While, the definition of $T_{((x_{1},y),(x_{2},y))}$ entails that
\begin{equation*}
T_{((x_{1},y),(x_{2},y))}(x_{1}-x_{2})=u(x_{1},y,\tau)-u(x_{2},y,\tau),
\end{equation*}
which implies (\ref{cocycle-2}) inductively.

Now, for the bundle map $(F_{1},T)$ on $K_{1}\times X$, we utilize Proposition \ref{prop ES-contin} to obtain the constants $M>0$, $0<\gamma<1$ and vectors $l_{((x_{1},y),(x_{2},y))}\in C_{s}^{*}$, $v_{((x_{1},y),(x_{2},y))}\in{\rm Int}C$ for any $((x_{1},y),(x_{2},y))\in K_{1}$, such that properties (i)-(v) in Proposition \ref{prop ES-contin} hold.

Due to the assumption in this Proposition, (\ref{cocycle-2}) entails that
\begin{equation*}
T_{((x_{3},\tilde{y}),(x_{4},\tilde{y}))}^{n}(x_{3}-x_{4})=u(x_{3},\tilde{y},n\tau)-u(x_{4},\tilde{y},n\tau)\notin\pm C,
\end{equation*}
for any $n\geq1$. Together with Proposition \ref{prop ES-contin}(iv), this implies that
\begin{equation}\label{l(x3,x4)(x3-x4)=0}
l_{((x_{3},\tilde{y}),(x_{4},\tilde{y}))}(x_{3}-x_{4})=0.
\end{equation}
Choose an $\epsilon_{0}>0$ so small that $\gamma e^{\epsilon_{0}\tau}<1$. Since $K$ is linearly stable, for such $\epsilon_{0}$, it follows from Proposition \ref{Prop norm growth control} that there is a $C_{\epsilon_{0}}>0$ such that
\begin{equation}\label{epsilon 0}
\|u_{x}(x,y,t)\|\leq C_{\epsilon_{0}}e^{\epsilon_{0}t}, \ \ \text{ for all } t\geq0 \text{ and } (x,y)\in K.
\end{equation}
We further choose an $\epsilon>0$ so small that
\begin{equation}\label{gamma,epsilon0,epsilon}
\gamma e^{\epsilon_{0}\tau}(1+\epsilon)<1.
\end{equation}
For such $\epsilon>0$, by Proposition \ref{prop ES-contin}(v), there exists a constant $\delta_{1}>0$ such that the estimate (\ref{3.1}) holds.
Let an integer $n_{0}\geq1$ be such that
\begin{equation}\label{the choice of n0}
C_{\epsilon_{0}}M(\gamma e^{\epsilon_{0}\tau}(1+\epsilon))^{n_{0}}<1,
\end{equation}
where $M$ is from the estimate (\ref{3.2}) in Proposition \ref{prop ES-contin}(iii).

Due to the continuity of $F_{1}$ on $K_{1}$, one can find some $\delta_{0}>0$ so small that
\begin{equation}\label{the choice of delta0}
d_{K_{1}}(F_{1}^{i}((x_{1},y),(x_{2},y)),F_{1}^{i}((x_{1}^{\prime},y^{\prime}),(x_{2}^{\prime},y^{\prime})))<\delta_{1},\text{ for any }0\leq i\leq n_{0},
\end{equation}
whenever $((x_{1},y),(x_{2},y)), ((x_{1}^{\prime},y^{\prime}),(x_{2}^{\prime},y^{\prime}))\in K_{1}$ with $d_{K_{1}}(((x_{1},y),(x_{2},y)),((x_{1}^{\prime},y^{\prime}),(x_{2}^{\prime},y^{\prime})))<\delta_{0}$.

Now, we \emph{claim} that, \emph{for any $(x_{3},\tilde{y}),(x_{4},\tilde{y})\in K$ with $\|x_{3}-x_{4}\|<\delta_{0}$ and $u(x_{3},\tilde{y},t),u(x_{4},\tilde{y},t)$ unordered for all $t\geq0$,}
\begin{equation}\label{discrete infty estimate}
\|u(x_{3},\tilde{y},i\tau)-u(x_{4},\tilde{y},i\tau)\|\leq C_{\epsilon_{0}}M(\gamma e^{\epsilon_{0}\tau}(1+\epsilon))^{i}\|x_{3}-x_{4}\|,\ \  \emph{ for any } i\geq1.
\end{equation}
In order to prove the claim, we first prove (\ref{discrete infty estimate}) for $1\leq i\leq n_{0}$. By taking $((x_{1},y),(x_{2},y))=((x_{3},\tilde{y}),(x_{4},\tilde{y}))$ and $((x_{1}^{\prime},y^{\prime}),(x_{2}^{\prime},y^{\prime}))=((x_{3},\tilde{y}),(x_{3},\tilde{y}))$ in (\ref{the choice of delta0}), we have
\begin{equation}\label{0-n0initial}
d_{K_{1}}(F_{1}^{i}((x_{3},\tilde{y}),(x_{4},\tilde{y})),F_{1}^{i}((x_{3},\tilde{y}),(x_{3},\tilde{y})))<\delta_{1},\quad 0\leq i\leq n_{0}.
\end{equation}
By virtue of (\ref{3.1}) in Proposition \ref{prop ES-contin}(v), one has
\begin{equation}\label{1+epsilon m0-n0}
\|T_{((x_{3},\tilde{y}),(x_{4},\tilde{y}))}^{i}v_{((x_{3},\tilde{y}),(x_{4},\tilde{y}))}\|\leq (1+\epsilon)^{i}\|T_{((x_{3},\tilde{y}),(x_{3},\tilde{y}))}^{i}v_{((x_{3},\tilde{y}),(x_{3},\tilde{y}))}\|,\quad 1\leq i\leq n_{0}.
\end{equation}
Therefore, for $1\leq i \leq n_{0}$,
\begin{alignat*}{2}\label{m0-n0}
\|u(x_{3},\tilde{y},i\tau)-u(x_{4},\tilde{y},i\tau)\|&\quad\overset{(\ref{cocycle-2})}{=}\|T_{((x_{3},\tilde{y}),(x_{4},\tilde{y}))}^{i}(x_{3}-x_{4})\|\\
&\overset{(\ref{l(x3,x4)(x3-x4)=0})+(\ref{3.2})}{\leq}M{\gamma}^{i}\|T_{((x_{3},\tilde{y}),(x_{4},\tilde{y}))}^{i}v_{((x_{3},\tilde{y}),(x_{4},\tilde{y}))}\|\cdot
\|x_{3}-x_{4}\|\\
&\ \ \overset{(\ref{1+epsilon m0-n0})}{\leq} M(\gamma(1+\epsilon))^{i}\|T_{((x_{3},\tilde{y}),(x_{3},\tilde{y}))}^{i}v_{((x_{3},\tilde{y}),(x_{3},\tilde{y}))}\|\cdot\|x_{3}-x_{4}\|\\
&\quad\overset{(\ref{cocycle-1})}{\leq}M(\gamma(1+\epsilon))^{i}\|u_{x}(x_{3},\tilde{y},i\tau)\|\cdot\|x_{3}-x_{4}\|\\
&\quad\overset{(\ref{epsilon 0})}{\leq}C_{\epsilon_{0}}M(\gamma e^{\epsilon_{0}\tau}(1+\epsilon))^{i}\|x_{3}-x_{4}\|.\tag{3.12}
\end{alignat*}
Next, we will prove (\ref{discrete infty estimate}) for $1\leq i\leq 2n_{0}$.
Choose $i=n_{0}$ in (\ref{m0-n0}). Then, together with (\ref{the choice of n0}), $\|u(x_{3},\tilde{y},n_{0}\tau)-u(x_{4},\tilde{y},n_{0}\tau)\|<\|x_{3}-x_{4}\|<\delta_{0}$. Hence, $$d_{K_{1}}(F_{1}^{n_{0}}((x_{3},\tilde{y}),(x_{4},\tilde{y})),F_{1}^{n_{0}}((x_{3},\tilde{y}),(x_{3},\tilde{y})))<\delta_{0}.$$
So, we again take
$$((x_{1},y),(x_{2},y))=F_{1}^{n_{0}}((x_{3},\tilde{y}),(x_{4},\tilde{y}))$$
and
$$((x_{1}^{\prime},y^{\prime}),(x_{2}^{\prime},y^{\prime}))=F_{1}^{n_{0}}((x_{3},\tilde{y}),(x_{3},\tilde{y})))$$
in (\ref{the choice of delta0}), and obtain $d_{K_{1}}(F_{1}^{i}((x_{3},\tilde{y}),(x_{4},\tilde{y})),F_{1}^{i}((x_{3},\tilde{y}),(x_{3},\tilde{y})))<\delta_{1}$, for $n_{0}\leq i\leq 2n_{0}$. Together with (\ref{0-n0initial}), we have $d_{K_{1}}(F_{1}^{i}((x_{3},\tilde{y}),(x_{4},\tilde{y})),F_{1}^{i}((x_{3},\tilde{y}),(x_{3},\tilde{y})))<\delta_{1}$ for any $0\leq i\leq 2n_{0}$. Again, by (\ref{3.1}) in Proposition \ref{prop ES-contin}(v), one has
\begin{equation*}\label{1+epsilon m0-2n0}
\|T_{((x_{3},\tilde{y}),(x_{4},\tilde{y}))}^{i}v_{((x_{3},\tilde{y}),(x_{4},\tilde{y}))}\|\leq (1+\epsilon)^{i}\|T_{((x_{3},\tilde{y}),(x_{3},\tilde{y}))}^{i}v_{((x_{3},\tilde{y}),(x_{3},\tilde{y}))}\|,\quad 1\leq i\leq 2n_{0}.\tag{3.13}
\end{equation*}
Therefore, for $1\leq i\leq 2n_{0}$,
\begin{alignat*}{2}
\|u(x_{3},\tilde{y},i\tau)-u(x_{4},\tilde{y},i\tau)\|&\ \,=\|T_{((x_{3},\tilde{y}),(x_{4},\tilde{y}))}^{i}(x_{3}-x_{4})\|\\
&\ \,\leq M{\gamma}^{i}\|T_{((x_{3},\tilde{y}),(x_{4},\tilde{y}))}^{i}v_{((x_{3},\tilde{y}),(x_{4},\tilde{y}))}\|\cdot
\|x_{3}-x_{4}\|\\
&\overset{(\ref{1+epsilon m0-2n0})}{\leq} M(\gamma(1+\epsilon))^{i}\|T_{((x_{3},\tilde{y}),(x_{3},\tilde{y}))}^{i}v_{((x_{3},\tilde{y}),(x_{3},\tilde{y}))}\|\cdot\|x_{3}-x_{4}\|\\
&\ \,\leq M(\gamma(1+\epsilon))^{i}\|u_{x}(x_{3},\tilde{y},i\tau)\|\cdot\|x_{3}-x_{4}\|\\
&\ \,\leq C_{\epsilon_{0}}M(\gamma e^{\epsilon_{0}\tau}(1+\epsilon))^{i}\|x_{3}-x_{4}\|.
\end{alignat*}
Inductively, we can repeat the arguments and prove that (\ref{discrete infty estimate}) is satisfied for all $i\geq1$. Thus, we have proved  the claim.

By virtue of the claim, we obtain $\|u(x_{3},\tilde{y},i\tau)-u(x_{4},\tilde{y},i\tau)\|$ $\to 0$, as $i\to+\infty$. Now, we show that $\|u(x_{3},\tilde{y},t)-u(x_{4},\tilde{y},t)\|\to 0$ as $t\to+\infty$. To this end, for any $\epsilon^{\prime}>0$, it follows from the uniform continuity of $u$ on $K\times [0,\tau]$ that, there exists $\delta^{\prime}>0$ such that $\|u(x_{1},y_{1},t_{1})-u(x_{2},y_{2},t_{2})\|<\epsilon^{\prime}$, for $(x_{1},y_{1},t_{1}),(x_{2},y_{2},t_{2})\in K\times [0,\tau]$ with $d_{K\times [0,\tau]}((x_{1},y_{1},t_{1}),(x_{2},y_{2},t_{2}))<\delta^{\prime}$. For any $t>0$, write $t=l\tau+\alpha$, $l\in\mathbb{N}$ and $\alpha\in[0,\tau]$. It follows from the claim that, there exists an integer $N>0$ such that $\|u(x_{3},\tilde{y},i\tau)-u(x_{4},\tilde{y},i\tau)\|<\delta^{\prime}$, for any $i\geq N$. Therefore, $\|u(x_{3},\tilde{y},t)-u(x_{4},\tilde{y},t)\|$=$\|u(u(x_{3},\tilde{y},l\tau),\tilde{y}\cdot l\tau,\alpha)
-u(u(x_{4},\tilde{y},l\tau),\tilde{y}\cdot l\tau,\alpha)\|<\epsilon^{\prime}$, for any $t\geq N\tau$. Thus, we have obtained (\ref{infty estimate}), which completes the proof.
\end{proof}

\begin{prop}\label{Prop substitute finite}
Let $K$ be as in Theorem \ref{Theorem main th}. If $(x_{3},\tilde{y}),(x_{4},\tilde{y})\in K$ satisfying $x_{3}-x_{4}\notin\pm C$, then the pair $(x_{3},\tilde{y})$ and $(x_{4},\tilde{y})$ are negatively distal.
\end{prop}

\begin{proof}
Suppose on the contrary that there exists a sequence $t_{n}\to-\infty$ such that $\|u(x_{3},\tilde{y},t_{n})-u(x_{4},\tilde{y},t_{n})\|\to0$ as $t_{n}\to-\infty$. Then, we will obtain a contradiction by showing that $x_{3}=x_{4}$. To this purpose, let $\epsilon>0$ be in (\ref{gamma,epsilon0,epsilon}). For such $\epsilon>0$, let $\delta_{1}>0$ be obtained in Proposition \ref{prop ES-contin}(v) (for the bundle map $(F_{1},T)$ on $K_{1}\times X$).

Now, we \emph{claim} that, \emph{for any $0<\delta<\delta_{1}$, there exists $t_{\delta}\in [-\tau,0]$, such that}
\begin{equation*}\label{x3-x4less than delta}
\|u(x_{3},\tilde{y},t_{\delta})-u(x_{4},\tilde{y},t_{\delta})\|<\delta.\tag{3.14}
\end{equation*}
Before we prove the claim, we show that how it implies that $x_{3}=x_{4}$. Suppose that $x_{3}\neq x_{4}$. Let $0<\epsilon^{\prime}=\|x_{3}-x_{4}\|$. Noticing that $u$ is uniformly continuous on $K\times [0,\tau]$, there exists $\delta^{\prime}>0$ such that $\|u(x_{1},y_{1},t_{1})-u(x_{2},y_{2},t_{2})\|<\epsilon^{\prime}$, whenever $(x_{1},y_{1},t_{1}),(x_{2},y_{2},t_{2})\in K\times [0,\tau]$ with $d_{K\times [0,\tau]}((x_{1},y_{1},t_{1}),(x_{2},y_{2},t_{2}))<\delta^{\prime}$. For any $0<\delta<\min\{\delta^{\prime},\delta_{1}\}$, it follows from the claim that there exists $t_{\delta}\in [-\tau,0]$, such that $\|u(x_{3},\tilde{y},t_{\delta})-u(x_{4},\tilde{y},t_{\delta})\|<\delta$. This implies $\|x_{3}-x_{4}\|=\|u(u(x_{3},\tilde{y},t_{\delta}),y\cdot t_{\delta},-t_{\delta})-u(u(x_{4},\tilde{y},t_{\delta}),y\cdot t_{\delta},-t_{\delta})\|<\epsilon^{\prime}$, which contradicts $\|x_{3}-x_{4}\|=\epsilon^{\prime}$.

Now, we focus on the proof of the claim. By virtue of Proposition \ref{prop ES-contin}(i), the set $V_{K_{1}}\triangleq\{v_{((x_{1},y),(x_{2},y))}:((x_{1},y),(x_{2},y))\in K_{1}\}\subset{\rm Int}C$ is compact. So, there exists $\epsilon_{1}>0$ such that
 \begin{equation*}\label{epsilon 1}
\{v\in X:\ d(v,V_{K_{1}})<\epsilon_{1}\}\subset{\rm Int}C,\tag{3.15}
\end{equation*}
where $d(v,V_{K_{1}})=\mathop{\inf}\limits_{w \in V_{K_{1}}}d(v,w)$. We decompose $u(x_{3},\tilde{y},t_{n})-u(x_{4},\tilde{y},t_{n})$ as
\begin{equation*}\label{decompose un}
u(x_{3},\tilde{y},t_{n})-u(x_{4},\tilde{y},t_{n})=c_{n}v_{(\Pi(x_{3},\tilde{y},t_{n}),\Pi(x_{4},\tilde{y},t_{n}))}+w_{n},\tag{3.16}
\end{equation*}
where $c_{n}=\frac{l_{(\Pi(x_{3},\tilde{y},t_{n}),\Pi(x_{4},\tilde{y},t_{n}))}(u(x_{3},\tilde{y},t_{n})-u(x_{4},\tilde{y},t_{n}))}
{l_{(\Pi(x_{3},\tilde{y},t_{n}),\Pi(x_{4},\tilde{y},t_{n}))}(v_{(\Pi(x_{3},\tilde{y},t_{n}),\Pi(x_{4},\tilde{y},t_{n}))})}$ and $l_{(\Pi(x_{3},\tilde{y},t_{n}),\Pi(x_{4},\tilde{y},t_{n}))}(w_{n})=0$. For each $n\geq1$, we write $-t_{n}=k_{n}\tau+\alpha_{n}$ with $k_{n}\in\mathbb{N}$, $\alpha_{n}\in[0,\tau)$, $n=1,2,\cdots$. We assert that
\begin{equation*}\label{cn control by gamma}
|c_{n}|\leq {\epsilon_{1}}^{-1}M\gamma^{k_{n}}\|w_{n}\|,\quad\text{for}\quad n\geq1,\tag{3.17}
\end{equation*}
where $M$ and $\gamma$ are from (\ref{3.2}) in Proposition \ref{prop ES-contin}(iii).
In fact, if $c_{n}=0$, then we've done. If $c_{n}\neq0$, then by (\ref{cocycle-2}) and (\ref{decompose un}), we write
\begin{alignat*}{2}
u(x_{3},\tilde{y},-\alpha_{n})-u(x_{4},\tilde{y},-\alpha_{n})
=\,&T_{(\Pi(x_{3},\tilde{y},t_{n}),\Pi(x_{4},\tilde{y},t_{n}))}^{k_{n}}(u(x_{3},\tilde{y},t_{n})-u(x_{4},\tilde{y},t_{n}))\\
=\,&T_{(\Pi(x_{3},\tilde{y},t_{n}),\Pi(x_{4},\tilde{y},t_{n}))}^{k_{n}}(c_{n}v_{(\Pi(x_{3},\tilde{y},t_{n}),\Pi(x_{4},\tilde{y},t_{n}))}+w_{n}).
\end{alignat*}
Noticing that $x_{3}-x_{4}\notin\pm C$, one has $u(x_{3},\tilde{y},-\alpha_{n})-u(x_{4},\tilde{y},-\alpha_{n})\notin\pm C$.
So,
$$\frac{T_{(\Pi(x_{3},\tilde{y},t_{n}),\Pi(x_{4},\tilde{y},t_{n}))}^{k_{n}}v_{(\Pi(x_{3},\tilde{y},t_{n}),\Pi(x_{4},\tilde{y},t_{n}))}}
{\|T_{(\Pi(x_{3},\tilde{y},t_{n}),\Pi(x_{4},\tilde{y},t_{n}))}^{k_{n}}v_{(\Pi(x_{3},\tilde{y},t_{n}),\Pi(x_{4},\tilde{y},t_{n}))}\|}+
\frac{T_{(\Pi(x_{3},\tilde{y},t_{n}),\Pi(x_{4},\tilde{y},t_{n}))}^{k_{n}}w_{n}}
{c_{n}\|T_{(\Pi(x_{3},\tilde{y},t_{n}),\Pi(x_{4},\tilde{y},t_{n}))}^{k_{n}}v_{(\Pi(x_{3},\tilde{y},t_{n}),\Pi(x_{4},\tilde{y},t_{n}))}\|}\notin\pm C.$$
Since
$T_{(\Pi(x_{3},\tilde{y},t_{n}),\Pi(x_{4},\tilde{y},t_{n}))}^{k_{n}}v_{(\Pi(x_{3},\tilde{y},t_{n}),\Pi(x_{4},\tilde{y},t_{n}))}\in V_{K_{1}}$,
(\ref{epsilon 1}) implies that
$$|c_{n}|\leq \frac{\|T_{(\Pi(x_{3},\tilde{y},t_{n}),\Pi(x_{4},\tilde{y},t_{n}))}^{k_{n}}w_{n}\|}
{\epsilon_{1}\|T_{(\Pi(x_{3},\tilde{y},t_{n}),\Pi(x_{4},\tilde{y},t_{n}))}^{k_{n}}v_{(\Pi(x_{3},\tilde{y},t_{n}),\Pi(x_{4},\tilde{y},t_{n}))}\|}.$$
Together with (\ref{3.2}) in Proposition \ref{prop ES-contin}(iii), we obtain that $|c_{n}|\leq{\epsilon_{1}}^{-1}M\gamma^{k_{n}}\|w_{n}\|$ for $n\geq1$. Thus, we have proved the assertion of (\ref{cn control by gamma}).

Since $t_{n}\to-\infty$, we can choose an integer $N_{1}>0$ such that $M\gamma^{k_{n}}<\frac{{\epsilon_{1}}}{2}$ for any $n\geq N_{1}$. By (\ref{decompose un})-(\ref{cn control by gamma}),
\begin{equation*}\label{wn control by un-un}
\|u(x_{3},\tilde{y},t_{n})-u(x_{4},\tilde{y},t_{n})\|>\frac{1}{2}\|w_{n}\|,\quad\text{for}\quad n\geq N_{1}.\tag{3.18}
\end{equation*}
Fix an integer $n_{0}\geq1$ such that
\begin{equation*}\label{the choice of n0'}
2\max\{{\epsilon_{1}}^{-1},1\}C_{\epsilon_{0}}M(\gamma e^{\epsilon_{0}\tau}(1+\epsilon))^{n_{0}}<\frac{1}{3}.\tag{3.19}
\end{equation*}
For any $\delta\in(0,\delta_{1})$ in (\ref{x3-x4less than delta}), due to the continuity of $F_{1}$ on $K_{1}$, one can choose $\delta_{0}>0$ so small that
\begin{equation*}\label{the choice of delta0'}
d_{K_{1}}(F_{1}^{i}((x_{1},y),(x_{2},y)),F_{1}^{i}((x_{1}^{\prime},y^{\prime}),(x_{2}^{\prime},y^{\prime})))<\delta,\ \text{ for }0\leq i\leq n_{0},\tag{3.20}
\end{equation*}
whenever $((x_{1},y),(x_{2},y)), ((x_{1}^{\prime},y^{\prime}),(x_{2}^{\prime},y^{\prime}))\in K_{1}$ with $d_{K_{1}}(((x_{1},y),(x_{2},y)),((x_{1}^{\prime},y^{\prime}),(x_{2}^{\prime},y^{\prime})))<\delta_{0}$.

Recall that $\|u(x_{3},\tilde{y},t_{n})-u(x_{4},\tilde{y},t_{n})\|\to0$ as $n\to+\infty$. Then there exists an integer $N_{2}>0$, such that
\begin{equation*}\label{initial less than delta0}
t_{n}<-2n_{0}\tau\quad\text{and}\quad\|u(x_{3},\tilde{y},t_{n})-u(x_{4},\tilde{y},t_{n})\|<\delta_{0},\quad\text{for}\quad n\geq N_{2}.\tag{3.21}
\end{equation*}
In other words, $$d_{K_{1}}((\Pi(x_{3},\tilde{y},t_{n}),\Pi(x_{4},\tilde{y},t_{n})),(\Pi(x_{3},\tilde{y},t_{n}),\Pi(x_{3},\tilde{y},t_{n})))<\delta_{0},\quad\text{ for}\quad n\geq N_{2}.$$
So, by taking in (\ref{the choice of delta0'}) $((x_{1},y),(x_{2},y))=(\Pi(x_{3},\tilde{y},t_{n}),\Pi(x_{4},\tilde{y},t_{n}))$ and $((x_{1}^{\prime},y^{\prime}),(x_{2}^{\prime},y^{\prime}))=(\Pi(x_{3},\tilde{y},t_{n}),\Pi(x_{3},\tilde{y},t_{n}))$
for some $n>\max\{N_{1},N_{2}\}$, we have
\begin{equation*}\label{0-n0initial'}
d_{K_{1}}(F_{1}^{i}(\Pi(x_{3},\tilde{y},t_{n}),\Pi(x_{4},\tilde{y},t_{n})),F_{1}^{i}(\Pi(x_{3},\tilde{y},t_{n}),\Pi(x_{3},\tilde{y},t_{n})))<\delta,\quad 0\leq i\leq n_{0}.\tag{3.22}
\end{equation*}
Together with (\ref{3.1}) in Proposition \ref{prop ES-contin}(v), we obtain
\begin{alignat*}{2}\label{1+epsilon m0-n0'}
&\|T_{(\Pi(x_{3},\tilde{y},t_{n}),\Pi(x_{4},\tilde{y},t_{n}))}^{i}v_{(\Pi(x_{3},\tilde{y},t_{n}),\Pi(x_{4},\tilde{y},t_{n}))}\|
\leq\,&(1+\epsilon)^{i}\|T_{(\Pi(x_{3},\tilde{y},t_{n}),\Pi(x_{3},\tilde{y},t_{n}))}^{i}v_{(\Pi(x_{3},\tilde{y},t_{n}),\Pi(x_{3},\tilde{y},t_{n}))}\|, \tag{3.23}
\end{alignat*}
for $1\leq i\leq n_{0}$.
Therefore, for $1\leq i \leq n_{0}$,
\begin{alignat*}{2}\label{m0-n0'}
&\|u(x_{3},\tilde{y},t_{n}+i\tau)-u(x_{4},\tilde{y},t_{n}+i\tau)\|\\
\overset{(\ref{cocycle-2})}{=}&\|T_{(\Pi(x_{3},\tilde{y},t_{n}),\Pi(x_{4},\tilde{y},t_{n}))}^{i}(u(x_{3},\tilde{y},t_{n})-u(x_{4},\tilde{y},t_{n}))\|\\
\overset{(\ref{decompose un})}{=}&\|T_{(\Pi(x_{3},\tilde{y},t_{n}),\Pi(x_{4},\tilde{y},t_{n}))}^{i}(c_{n}v_{(\Pi(x_{3},\tilde{y},t_{n}),\Pi(x_{4},\tilde{y},t_{n}))}+w_{n})\|\\
\overset{(\ref{3.2})}{\leq}&(|c_{n}|+M\gamma^{i}\|w_{n}\|)\cdot\|T_{(\Pi(x_{3},\tilde{y},t_{n}),\Pi(x_{4},\tilde{y},t_{n}))}^{i}v_{(\Pi(x_{3},\tilde{y},t_{n}),\Pi(x_{4},\tilde{y},t_{n}))}\|\\
\overset{(\ref{cn control by gamma})}{\leq}&({\epsilon_{1}}^{-1}M\gamma^{k_{n}}+M\gamma^{i})\|w_{n}\|\cdot\|T_{(\Pi(x_{3},\tilde{y},t_{n}),\Pi(x_{4},\tilde{y},t_{n}))}^{i}v_{(\Pi(x_{3},\tilde{y},t_{n}),\Pi(x_{4},\tilde{y},t_{n}))}\|\\
\overset{(\ref{1+epsilon m0-n0'})}{\leq}&({\epsilon_{1}}^{-1}M\gamma^{k_{n}}+M\gamma^{i})\|w_{n}\|\cdot(1+\epsilon)^{i}\|T_{(\Pi(x_{3},\tilde{y},t_{n}),\Pi(x_{3},\tilde{y},t_{n}))}^{i}v_{(\Pi(x_{3},\tilde{y},t_{n}),\Pi(x_{3},\tilde{y},t_{n}))}\|\\
\overset{(\ref{cocycle-1})}{\leq}&({\epsilon_{1}}^{-1}M\gamma^{k_{n}}+M\gamma^{i})\|w_{n}\|\cdot(1+\epsilon)^{i}\|u_{x}(\Pi(x_{3},\tilde{y},t_{n}),i\tau)\|\\
\overset{(\ref{wn control by un-un})}{\leq}&2({\epsilon_{1}}^{-1}M\gamma^{k_{n}}+M\gamma^{i})(1+\epsilon)^{i}\|u_{x}(\Pi(x_{3},\tilde{y},t_{n}),i\tau)\|\cdot\|u(x_{3},\tilde{y},t_{n})-u(x_{4},\tilde{y},t_{n})\|\\
\overset{(\ref{epsilon 0})}{\leq}&\left[2{\epsilon_{1}}^{-1}C_{\epsilon_{0}}M\gamma^{k_{n}}(e^{\epsilon_{0}\tau}(1+\epsilon))^{i}+2C_{\epsilon_{0}}M(\gamma e^{\epsilon_{0}\tau}(1+\epsilon))^{i}\right]\cdot\|u(x_{3},\tilde{y},t_{n})-u(x_{4},\tilde{y},t_{n})\|.\tag{3.24}
\end{alignat*}
Choose $i=n_{0}$ in (\ref{m0-n0'}). Then by (\ref{the choice of n0'}), (\ref{initial less than delta0}), we have
\begin{alignat*}{2}
&\|u(x_{3},\tilde{y},t_{n}+n_{0}\tau)-u(x_{4},\tilde{y},t_{n}+n_{0}\tau)\|<\frac{2}{3}\delta_{0},
\end{alignat*}
and hence, $$d_{K_{1}}(F_{1}^{n_{0}}(\Pi(x_{3},\tilde{y},t_{n}),\Pi(x_{4},\tilde{y},t_{n})),F_{1}^{n_{0}}(\Pi(x_{3},\tilde{y},t_{n}),\Pi(x_{3},\tilde{y},t_{n})))<\frac{2}{3}\delta_{0}<\delta_{0},$$
by which we take in (\ref{the choice of delta0'})
$$((x_{1},y),(x_{2},y))=F_{1}^{n_{0}}(\Pi(x_{3},\tilde{y},t_{n}),\Pi(x_{4},\tilde{y},t_{n})),$$
$$((x_{1}^{\prime},y^{\prime}),(x_{2}^{\prime},y^{\prime}))=F_{1}^{n_{0}}(\Pi(x_{3},\tilde{y},t_{n}),\Pi(x_{3},\tilde{y},t_{n})),$$
and obtain
\begin{equation*}
d_{K_{1}}(F_{1}^{i}(\Pi(x_{3},\tilde{y},t_{n}),\Pi(x_{4},\tilde{y},t_{n})),F_{1}^{i}(\Pi(x_{3},\tilde{y},t_{n}),\Pi(x_{3},\tilde{y},t_{n})))<\delta,
\end{equation*}
for $n_{0}\leq i\leq 2n_{0}$. So, together with (\ref{0-n0initial'}), we obtain
\begin{equation*}
d_{K_{1}}(F_{1}^{i}(\Pi(x_{3},\tilde{y},t_{n}),\Pi(x_{4},\tilde{y},t_{n})),F_{1}^{i}(\Pi(x_{3},\tilde{y},t_{n}),\Pi(x_{3},\tilde{y},t_{n})))<\delta,\ \ 0\leq i\leq 2n_{0}.
\end{equation*}
Again, by (\ref{3.1}) in Proposition \ref{prop ES-contin}(v), one has
\begin{alignat*}{2}\label{1+epsilon m0-2n0'}
&\|T_{(\Pi(x_{3},\tilde{y},t_{n}),\Pi(x_{4},\tilde{y},t_{n}))}^{i}v_{(\Pi(x_{3},\tilde{y},t_{n}),\Pi(x_{4},\tilde{y},t_{n}))}\|\\
\leq\,&(1+\epsilon)^{i}\|T_{(\Pi(x_{3},\tilde{y},t_{n}),\Pi(x_{3},\tilde{y},t_{n}))}^{i}v_{(\Pi(x_{3},\tilde{y},t_{n}),\Pi(x_{3},\tilde{y},t_{n}))}\|, \tag{3.25}
\end{alignat*}
for any $1\leq i\leq 2n_{0}$.
Therefore, similarly as the estimates in (\ref{m0-n0'}), we obtain from (\ref{1+epsilon m0-2n0'}) that
\begin{alignat*}{2}\label{m0-2n0'}
&\|u(x_{3},\tilde{y},t_{n}+i\tau)-u(x_{4},\tilde{y},t_{n}+i\tau)\|\\
\leq\,&\left[2{\epsilon_{1}}^{-1}C_{\epsilon_{0}}M\gamma^{k_{n}}(e^{\epsilon_{0}\tau}(1+\epsilon))^{i}+2C_{\epsilon_{0}}M(\gamma e^{\epsilon_{0}\tau}(1+\epsilon))^{i}\right]\cdot\|u(x_{3},\tilde{y},t_{n})-u(x_{4},\tilde{y},t_{n})\|,\tag{3.26}
\end{alignat*}
for any $1\leq i\leq 2n_{0}$.

Therefore, by repeating the same arguments, we obtain that
\begin{alignat*}{2}\label{discrete finite estimate}
&\|u(x_{3},\tilde{y},t_{n}+i\tau)-u(x_{4},\tilde{y},t_{n}+i\tau)\|\\
\leq\,&\left[2{\epsilon_{1}}^{-1}C_{\epsilon_{0}}M\gamma^{k_{n}}(e^{\epsilon_{0}\tau}(1+\epsilon))^{i}+2C_{\epsilon_{0}}M(\gamma e^{\epsilon_{0}\tau}(1+\epsilon))^{i}\right]\cdot\|u(x_{3},\tilde{y},t_{n})-u(x_{4},\tilde{y},t_{n})\|, \tag{3.27}
\end{alignat*}
for all $1\leq i\leq l_{n}\cdot n_{0}$, where the integer $l_{n}$ comes from the expression $-t_{n}=l_{n}\cdot n_{0}\tau+\beta_{n}$, with $\beta_{n}\in[0,n_{0}\tau)$. Clearly, $k_{n}\geq l_{n}\cdot n_{0}$ for $n\geq1$.

Let $i=l_{n}\cdot n_{0}$ in (\ref{discrete finite estimate}). Note that $k_{n}\geq l_{n}\cdot n_{0}$, again by (\ref{the choice of n0'}), (\ref{initial less than delta0}), we have
\begin{alignat*}{2}
&\|u(x_{3},\tilde{y},t_{n}+l_{n}\cdot n_{0}\tau)-u(x_{4},\tilde{y},t_{n}+l_{n}\cdot n_{0}\tau)\|\\
\leq\,&(2{\epsilon_{1}}^{-1}C_{\epsilon_{0}}M\gamma^{k_{n}}(e^{\epsilon_{0}\tau}(1+\epsilon))^{l_{n}\cdot n_{0}}+2C_{\epsilon_{0}}M(\gamma e^{\epsilon_{0}\tau}(1+\epsilon))^{l_{n}\cdot n_{0}})\delta_{0}\\
\leq\,&(2{\epsilon_{1}}^{-1}C_{\epsilon_{0}}M\gamma^{k_{n}}(e^{\epsilon_{0}\tau}(1+\epsilon))^{k_{n}}+2C_{\epsilon_{0}}M(\gamma e^{\epsilon_{0}\tau}(1+\epsilon))^{n_{0}})\delta_{0}\\
\leq\,&(2{\epsilon_{1}}^{-1}C_{\epsilon_{0}}M(\gamma e^{\epsilon_{0}\tau}(1+\epsilon))^{n_{0}}+2C_{\epsilon_{0}}M(\gamma e^{\epsilon_{0}\tau}(1+\epsilon))^{n_{0}})\delta_{0}<\frac{2}{3}\delta_{0},
\end{alignat*}
and hence,
$$d_{K_{1}}(F_{1}^{l_{n}\cdot n_{0}}(\Pi(x_{3},\tilde{y},t_{n}),\Pi(x_{4},\tilde{y},t_{n})),
F_{1}^{l_{n}\cdot n_{0}}(\Pi(x_{3},\tilde{y},t_{n}),\Pi(x_{3},\tilde{y},t_{n})))<\frac{2}{3}\delta_{0}<\delta_{0}.$$

Finally, again, we take in (\ref{the choice of delta0'}) $$((x_{1},y),(x_{2},y))=F_{1}^{l_{n}\cdot n_{0}}(\Pi(x_{3},\tilde{y},t_{n}),\Pi(x_{4},\tilde{y},t_{n})),$$
$$((x_{1}^{\prime},y^{\prime}),(x_{2}^{\prime},y^{\prime}))=F_{1}^{l_{n}\cdot n_{0}}(\Pi(x_{3},\tilde{y},t_{n}),\Pi(x_{3},\tilde{y},t_{n})),$$
and obtain
\begin{equation*}\label{less than delta}
d_{K_{1}}(F_{1}^{i}(\Pi(x_{3},\tilde{y},t_{n}),\Pi(x_{4},\tilde{y},t_{n})),F_{1}^{i}(\Pi(x_{3},\tilde{y},t_{n}),\Pi(x_{3},\tilde{y},t_{n})))<\delta,\tag{3.28}
\end{equation*}
for any $l_{n}\cdot n_{0}\leq i\leq(l_{n}+1)\cdot n_{0}$. In particular, one find an integer $i_{0}$ satisfying $l_{n}\cdot n_{0}\leq i_{0}\leq(l_{n}+1)\cdot n_{0}$ such that $t_{n}+i_{0}\tau\in [-\tau,0]$. Write $t_{\delta}=t_{n}+i_{0}\tau\in [-\tau,0]$. Note that $F_{1}^{i_{0}}(\Pi(x_{3},\tilde{y},t_{n}),\Pi(x_{j},\tilde{y},t_{n}))=(\Pi(x_{3},\tilde{y},t_{\delta}),\Pi(x_{j},\tilde{y},t_{\delta}))$, $j=3,4$. Then (\ref{less than delta}) directly implies that
$\|u(x_{3},\tilde{y},t_{\delta})-u(x_{4},\tilde{y},t_{\delta})\|<\delta$. Thus, we have proved the claim, which completes our proof.
\end{proof}

\begin{rmk}\label{remark3.1}
Proposition \ref{Prop substitute infty} and Proposition \ref{Prop substitute finite} play \emph{very crucial roles} in proving our main Theorem. For $C^{1,\alpha}$-smooth skew-product semiflows, these two Propositions were proved in Shen and Yi \cite[PartII, Lemma 4.6]{ShYi98}.
\end{rmk}

\begin{prop}\label{Prop P(K) equivalence P(K)=O(K)}
Let $K$ be as in Theorem \ref{Theorem main th}. Then

\textnormal{(a)} the proximal relation $P(K)$ on $K$ is an equivalence relation;

\textnormal{(b)} $P(K)=O(K)$.
\end{prop}

\begin{proof}
The proof of this proposition is analogous to that in \cite[PartII, Lemma 4.7 and Lemma 4.8]{ShYi98}. We give the detail for the sake of completeness.

(a) Since $(Y,\mathbb{R})$ is distal, $P(K)=\{((x_{1},y),(x_{2},y))\in K|(x_{1},y),(x_{2},y)\text{ are }proximal\}$. We only need to check the transitivity. Let $((x_{1},y),(x_{2},y)),((x_{2},y),(x_{3},y))\in P(K)$. One of the following alternatives must occur:

(i). There is a $t_{0}\geq0$ such that $(\Pi(x_{1},y,t_{0}),\Pi(x_{2},y,t_{0}))\in O(K)$ and $(\Pi(x_{2},y,t_{0}),$ $\Pi(x_{3},y,t_{0}))\in O(K)$.

(ii). There is a $t_{0}\geq0$ such that $((\Pi(x_{1},y,t_{0}),\Pi(x_{2},y,t_{0}))\in O(K)$ but $(\Pi(x_{2},y,t),\Pi(x_{3},y,t))$ $\notin O(K)$ for all $t\geq0$.

(iii). There is a $t_{0}\geq0$ such that $(\Pi(x_{2},y,t_{0}),\Pi(x_{3},y,t_{0}))\in O(K)$ but $(\Pi(x_{1},y,t),\Pi(x_{2},y,t))$ $\notin O(K)$ for all $t\geq0$.

(iv). For all $t\geq0$, $(\Pi(x_{1},y,t),\Pi(x_{2},y,t))\notin O(K)$, $(\Pi(x_{2},y,t),\Pi(x_{3},y,t))\notin O(K)$.

If (i) holds, then denote $(x_{i}^{*},y^{*})=\Pi(x_{i},y,t_{0})$, $i=1,2,3$. Let $y_{0}\in Y_{0}$ in Proposition \ref{Prop residual set unordered}(i). Since $(Y,\mathbb{R})$ is minimal, there exists a sequence $\{t_{n}\}$ such that $y^{*}\cdot t_{n}\to y_{0}$. By taking a subsequence, if necessary, we assume that $\Pi(x_{1}^{*},y^{*},t_{n})\to (\hat{x_{1}},y_{0})$, $\Pi(x_{2}^{*},y^{*},t_{n})\to (\hat{x_{2}},y_{0})$. Since $O(K)$ is a closed relation, $(\hat{x_{1}},y_{0})$ and $(\hat{x_{2}},y_{0})$ are ordered. By Proposition \ref{Prop residual set unordered}(i), one has $(\hat{x_{1}},y_{0})=(\hat{x_{2}},y_{0})$. Hence, $d_{K}(\Pi(x_{1}^{*},y^{*},t_{n}),\Pi(x_{2}^{*},y^{*},t_{n}))\to0$. Similarly, by taking a subsequence if necessary, we have $d_{K}(\Pi(x_{2}^{*},y^{*},t_{n}),\Pi(x_{3}^{*},y^{*},t_{n}))\to0$. Consequently, $d_{K}(\Pi(x_{1}^{*},y^{*},t_{n}),\Pi(x_{3}^{*},y^{*},t_{n}))\to0$, that is, $d_{K}(\Pi(x_{1},y,t_{n}+t_{0}),\Pi(x_{3},y,t_{n}+t_{0}))\to0$. So, $((x_{1},y),(x_{3},y))\in P(K)$.

If (ii) holds, then take $y_{0}\in Y_{0}$ in Proposition \ref{Prop residual set unordered}(i). Let $(x_{4},y_{0})\in K$, since $K$ is minimal, there exists a sequence $t_{n}\to\infty$ such that $y\cdot t_{n}\to y_{0}$. By repeating the same argument in (i), there is a subsequence, still denoted by $t_{n}$, such that $d_{K}(\Pi(x_{1},y,t_{n}),\Pi(x_{2},y,t_{n}))\to0$.

Since $((x_{2},y),(x_{3},y))\in P(K)$, there exists $t_{0}\in\mathbb{R}$ such that $d_{K}(\Pi(x_{2},y,t_{0}),\Pi(x_{3},y,t_{0}))<\delta_{0}$ ($\delta_{0}$ is in Proposition \ref{Prop substitute infty}). Then it follows from Proposition \ref{Prop substitute infty} that $\|u(x_{2},y,t)-u(x_{3},y,t)\|\to0$, as $t\to+\infty$. Therefore, we have $\|u(x_{1},y,t_{n})-u(x_{3},y,t_{n})\|\to0$, as $t\to+\infty$, that is, $((x_{1},y),(x_{3},y))\in P(K)$.

The proof of (iii) is analogous, we omit it.

Finally, if (iv) holds, then $(\Pi(x_{1},y,t),\Pi(x_{2},y,t))\notin O(K)$ for all $t\in \mathbb{R}$. Since $((x_{1},y),(x_{2},y))$ $\in P(K)$, there exists $\zeta\in\mathbb{R}$ such that $\|u(x_{1},y,\zeta)-u(x_{2},y,\zeta)\|$ is sufficiently small. It then follows from Proposition \ref{Prop substitute infty} that $\|u(x_{1},y,t)-u(x_{2},y,t)\|\to0$, as $t\to+\infty$. Similarly, we obtain $\|u(x_{2},y,t)-u(x_{3},y,t)\|\to0$, as $t\to+\infty$. Therefore, $\|u(x_{1},y,t)-u(x_{3},y,t)\|\to0$, that is, $((x_{1},y),(x_{3},y))\in P(K)$.

(b) By Proposition \ref{Prop residual set unordered}(ii), $O(K)\subset P(K)$. Now, we prove $P(K)\subset O(K)$. Suppose that $((x_{1},y),(x_{2},y))\in P(K)\backslash O(K)$. Then Proposition \ref{Prop substitute finite} implies that $(x_{1},y),(x_{2},y)$ are negatively distal. Therefore, $(x_{1},y),(x_{2},y)$ are both proximal and negatively distal. But this is impossible by Proposition \ref{prop P=P+=P-} and Proposition \ref{Prop P(K) equivalence P(K)=O(K)}(a).
\end{proof}

Now, we are ready to prove Theorem \ref{Theorem main th}.

\noindent$\emph{Proof of Theorem \ref{Theorem main th}.}$
By Proposition \ref{Prop P(K) equivalence P(K)=O(K)}(b), $P(K)=O(K)$ are invariant and closed. Let $\tilde{Y}=K/P(K)=K/O(K)$. Then, $(K,\mathbb{R})$ induces a flow $(\tilde{Y},\mathbb{R})$ by the invariance of $P(K)$. Clearly, $(\tilde{Y},\mathbb{R})$ is distal. Let $p:K\to Y; (x,y)\mapsto y$ be the natural projection. Denote by $\tilde{p}:\tilde{Y}\to Y; [(x,y)]\mapsto y$ the projection induced by $p$; and denote by $p^{*}:K\to \tilde{Y}=K/P(K)$ the natural projection to $\tilde{Y}$ as $p^{*}(x,y)=[(x,y)],(x,y)\in K$. So, $p=\tilde{p}\circ p^{*}$. By the closeness of $P(K)$, $\tilde{p}$ and $p^{*}$ are continuous. Let $Y_{0}$ be the residual set given by Proposition \ref{Prop residual set unordered}(i) and fix a $y_{0}\in Y_{0}$. Since Proposition \ref{Prop residual set unordered}(i) implies no two points on $p^{-1}(y_{0})$ are ordered, $card(p^{-1}(y_{0}))=card(\tilde{p}^{-1}(y_{0}))$. Now, if $card(\tilde{p}^{-1}(y_{0}))=\infty$, then there is an accumulation point $(x_{*},y_{0})\in p^{-1}(y_{0})$. Choose a $(x_{0},y_{0})\in p^{-1}(y_{0})$ such that $(x_{0},y_{0})\neq(x_{*},y_{0})$ and $\|x_{0}-x_{*}\|<\delta_{0}$, where $\delta_{0}$ is in Proposition \ref{Prop substitute infty}. Since $(x_{0},y_{0}),(x_{*},y_{0})$ are not ordered, Proposition \ref{Prop residual set unordered}(i) implies that $u(x_{0},y_{0},t),u(x_{*},y_{0},t)$ are not ordered for all $t\geq0$. Hence, by Proposition \ref{Prop substitute infty}, $\|u(x_{0},y_{0},t)-u(x_{*},y_{0},t)\|\to0$ as $t\to+\infty$, which implies that $(x_{0},y_{0})$ and $(x_{*},y_{0})$ are proximal, a contradiction to Proposition \ref{Prop P(K) equivalence P(K)=O(K)}(b). Thus, there is an integer $N\geq1$ such that $card(\tilde{p}^{-1}(y_{0}))=N$. By Proposition \ref{prop N-1}, $\tilde{p}$ is an $N$-1 extension.

Next, for any $y\in Y_{0}$ and any $[(x,y)]\in{\tilde{p}}^{-1}(y)$, since $(x^{\prime},y)$ and $(x,y)$ are not ordered for any $(x^{\prime},y)\neq(x,y)$, one has ${p^{*}}^{-1}[(x,y)]=\{(x,y)\}$, that is, $card({p^{*}}^{-1}[(x,y)])=1$. Since $Y_{0}$ is residual in $Y$, one has $\tilde{Y_{0}}=\{[(x,y)]\in\tilde{p}^{-1}(y)|y\in Y_{0}\}$ is residual in $\tilde{Y}$. Therefore, $p^{*}:(K,\mathbb{R})\to (\tilde{Y},\mathbb{R})$ is an almost 1-1 extension.

Now, if $(Y,\mathbb{R})$ is almost periodic, then by Proposition \ref{prop N-1}, $(\tilde{Y},\mathbb{R})$ is also almost periodic, and $(K,\mathbb{R})$ is almost automorphic, since $p^{*}$ is an almost 1-1 extension.\hfill $\square$
\vskip 2mm

Before ending this paper, we give the following two additional remarks.

\begin{rmk}\label{remark3.2}
Under the $C^{1}$-smoothness assumption, we show in Theorem \ref{Theorem main th} the almost automorphy of linearly stable minimal set for strongly monotone skew-product semiflows. The result was obtained by Shen and Yi \cite[PartII, Theorem 4.5]{ShYi98} for $C^{1,\alpha}$-systems. As a consequence, one can apply our theoretical result (Theorem \ref{Theorem main th}) to obtain all the same results in \cite[Part III]{ShYi98}, under the lower $C^{1}$-regularity (instead of $C^{1,\alpha}$), for time-almost periodic differential equations, including ODEs, parabolic equations and delay equations.
\end{rmk}

\begin{rmk}\label{remark3.3}
It deserves to point out that, under $C^{1,\alpha}$-smoothness, Novo et al. \cite{obaya13} showed that assumptions (i)-(ii) in Theorem \ref{Theorem main th} imply that $K$ admits a flow extension automatically. It is an interesting question whether it remains true under the weaker $C^{1}$-smoothness hypothesis.
\end{rmk}

\end{document}